\documentclass[reqno,12pt]{amsart}

\usepackage{epsf}
\usepackage{graphics}
\usepackage{graphicx}
\usepackage{amssymb}
\usepackage{amsmath}

\date{}

\theoremstyle{plain}
\newtheorem{theorem}{Theorem}
\newtheorem{corollary}{Corollary}
\newtheorem{proposition}{Proposition}

\newtheorem{rem}{Remark}

\theoremstyle{definition}

\theoremstyle{remark}

\def\N{{\mathbb N}}

\def\Q{{\mathbb Q}}

\title{Signature spectrum of positive braids}

\author{Sebastian Baader}

\begin{document}

\begin{abstract} We derive a lower bound for all Levine-Tristram signatures of positive braid links, linear in terms of the first Betti number. As a consequence, we obtain upper and lower bounds on the ratio of fixed pairs of Levine-Tristram signature invariants, valid uniformly on all positive braid monoids.
\end{abstract}

%\dedicatory{}

\maketitle

\section{Introduction}

Around 40 years ago, Rudolph showed that knotted positive braid links have strictly positive signature invariant, and concluded that these links are not slice~\cite{R}. Feller improved the strict positivity to a lower bound which is linear in the first Betti number $b_1(\beta)$ of the canonical Seifert surface associated with a positive braid $\beta$; the best known published lower bound is $\sigma(\beta) \geq \frac{1}{8} b_1(\beta)$ (see~\cite{F1} for this bound, and~\cite{F2} for the original bound, both by Feller). In this note, we derive a linear lower bound for the family of Levine-Tristram signature invariants $\sigma_\omega(L)$, parametrised by $\omega \in S^1$, defined from any Seifert matrix $V$ of a link $L$ as the signature of the Hermitian form
$$M_\omega=(1-\omega)V+(1-\bar{\omega})V^T.$$
For the class of links considered here -- closures of non-split positive braids -- this form is non-degenerate, provided $\omega$ is not an algebraic number, i.e. $\omega \in S^1 \setminus \bar{\Q}$. In particular, the signature is just the difference of the numbers of positive and negative eigenvalues of the matrix $M_\omega$
(see~\cite{L,Tri} for the original definition, and~\cite{C} for an excellent survey on the Levine-Tristram signature invariants).
We use the notation $B_n^+$ for the monoid of positive braids on $n$ strands, and $\sigma_\omega(\beta)$ for the Levine-Tristram signature invariant of the closure of a braid $\beta \in B_n^+$.

\begin{theorem}
\label{linear}
For all $\omega \in S^1 \setminus \bar{\Q}$, there exists a constant $c>0$, such that for all $n \in \N$:
$$\liminf \left\{\frac{\sigma_{\omega}(\beta)}{b_1(\beta)} \middle| \beta \in B_n^+ \right\}\geq c.$$
\end{theorem}
 
The limit inferior is unavoidable, since the piecewise constant signature function of a knot always starts off with a zero segment. While each individual of these bounds may not appear that impressive, we would like to record the following consequence of Theorem~\ref{linear} and the inequalities $\sigma_\omega(\beta) \leq 2b_1(\beta)$, valid for all $\omega \in S^1$, and for all positive braids $\beta$.

\begin{corollary}
Let $\omega_1,\omega_2 \in S^1 \setminus \bar{\Q}$. There exist constants $0<a<b$, so that the following inequalities hold for all $n \geq 2$:
$$a \leq \liminf \left\{\frac{\sigma_{\omega_1}(\beta)}{\sigma_{\omega_2}(\beta)} \middle| \beta \in B_n^+\right\} \leq \limsup
\left\{\frac{\sigma_{\omega_1}(\beta)}{\sigma_{\omega_2}(\beta)} \middle| \beta \in B_n^+\right\}
\leq b.$$
\end{corollary}

The proof of Theorem~\ref{linear} has two parts, depending on the real part of $\omega$. The case $\text{Re}(\omega)>-\frac{1}{2}$ is a direct consequence of Gambaudo and Ghys' results \cite{GG}. The case $\text{Re}(\omega)<-\frac{1}{2}$ involves a special kind of plumbing operation on positive 3-braids. We would like to stress that our proof breaks down in the limit points $\omega=1$ and $-1$. In particular, our argument does not recover Feller's result for the classical signature invariant $\sigma=\sigma_{-1}$.
For a similar reason, our argument does not imply that the zeros of the Alexander polynomial of random positive braids equidistribute on~$S^1$. This interesting problem remains open.

\section{Gambaudo-Ghys and $Re(\omega)>-\frac{1}{2}$}

In their influential article on braids, Gambaudo and Ghys derived that the restriction of the Levine-Tristram signature function $\sigma_{\omega=e^{2 \pi i \theta}}$ to every fixed braid group $B_k$ with $k \geq 3$ is quasi-linear on the interval $\theta \in (0,\frac{1}{k})$. More precisely, the function $\theta \mapsto \sigma_{\omega=e^{2 \pi i \theta}}(\beta)$ is linear with slope $2 \text{lk}(\beta)$, up to a constant error $k-1$ (see Corollary~4.4. in~\cite{GG} for the case $k=3$, and the remark just before Corollary~4.4. for general $k \geq 3$). In the case of positive braids $\beta \in B_k^+$, the linking number $\text{lk}(\beta)$ coincides with the first Betti number $b_1(\beta)$, up to the same constant error. As a consequence, we obtain the following estimate for all 
$\theta \in (\frac{1}{k+1},\frac{1}{k}):$
$$\liminf \left\{\frac{\sigma_{e^{2 \pi i \theta}}(\beta)}{b_1(\beta)} \middle| \beta \in B_k^+ \right\} > \frac{2}{k+1}.$$
This settles Theorem~\ref{linear} for all $\omega \in S^1 \setminus \bar{\Q}$ with $\text{Re}(\omega)>\text{Re}(\zeta_3)=-\frac{1}{2}$, by Feller's method described in~\cite{F1}:
let the above limit inferior be $\frac{2c}{k+1}$ with $c>1$, and let $\alpha \in B_n^+$, for any $n \geq k$. We can turn $\alpha$ into a finite union of positive $k$-braids~$\bar{\alpha}$, by smoothing a ratio of at most $\frac{1}{k}$ crossings of~$\alpha$. Then, up to a fixed constant error:
\begin{equation*} %\label{}
\begin{split}
\sigma_\omega(\alpha) & \geq \sigma_\omega(\bar{\alpha})-\frac{1}{k} b_1(\alpha) \\
 & \geq \frac{2c}{k+1} b_1(\bar{\alpha})-\frac{1}{k} b_1(\alpha) \\
 & \geq \frac{2c}{k+1} \left(1-\frac{1}{k}\right) b_1(\alpha)-\frac{1}{k} b_1(\alpha) \\
 & =\frac{2c(k-1)-(k+1)}{k(k+1)} b_1(\alpha).
\end{split}
\end{equation*}
The last constant is independent of~$n$ and strictly positive, since $c>1$ and $k \geq 3$.

\section{Pentafoil plumbing and $Re(\omega)<-\frac{1}{2}$}

The purpose of this section is to derive a lower bound of the form
$$\liminf \left\{\frac{\sigma_{\omega}(\beta)}{b_1(\beta)} \middle| \beta \in B_3^+ \right\} \geq \frac{c}{2},$$
for all $\omega \in S^1 \setminus \bar{\Q}$ with $\text{Re}(\omega)<-\frac{1}{2}$, with a constant $c>1$ depending on $\omega$. As explained at the end of the previous section, this settles Theorem~\ref{linear} for all $\omega \in S^1 \setminus \bar{\Q}$ with $\text{Re}(\omega)<-\frac{1}{2}$.

We will derive the above estimate by analysing the Garside normal form of braids, which is particularly easy for positive 3-braids~\cite{Ga}. A similar argument was used in the recent work of Tru\"ol, where the concordance invariant $\nu=\Upsilon(1)$ is determined for all 3-braid knots~\cite{Tru}.

\begin{proposition} \label{pentafoil}
Every positive 3-braid $\beta \in B_3^+$ can be obtained from a positive 3-braid $\beta_0 \in B_3^+$, of length at most nine, by successively inserting fourth or higher powers of a single generator, and possibly a second or third power of a single generator at the very end.
\end{proposition}

\begin{proof}
Up to conjugation, every positive 3-braid $\beta \in B_3^+$ can be written as
$$\beta=\Delta^N \sigma_1^{a_1} \sigma_2^{a_2} \sigma_1^{a_3} \ldots \sigma_2^{a_{2n}},$$
for some suitable natural numbers $N,a_i \in \N$, all $\geq 2$, except possibly $N$ and $a_1,a_2$, which can both be zero or one. Here $\Delta=\sigma_2 \sigma_1 \sigma_2=\sigma_1 \sigma_2 \sigma_1$ is the square root of the positive generator
$\Delta^2=\sigma_1 \sigma_2^2 \sigma_1 \sigma_2^2=\sigma_2 \sigma_1^2 \sigma_2 \sigma_1^2$ of the centre of~$B_3$. A little computation shows
$$\Delta^3=\sigma_2 \sigma_1^2 \sigma_2^2 \sigma_1 \sigma_2^3=\sigma_1 \sigma_2^2 \sigma_1^2 \sigma_2 \sigma_1^3,$$
$$\Delta^4=\sigma_1 \sigma_2^2 \sigma_1^2 \sigma_2^2 \sigma_1 \sigma_2^4=\sigma_2 \sigma_1^2 \sigma_2^2 \sigma_1^2 \sigma_2 \sigma_1^4.$$

This will be used in order to find high powers of a single generator in positive 3-braids. Indeed, if one of the following 5 conditions is satisfied, then up to conjugation, the braid $\beta$ can be written as a positive braid word that contains the fourth power of a generator. 
\begin{enumerate}
\item $N \geq 4$,
\item $N \geq 3$ and at least one of the $a_i$ is strictly positive,
\item $N \geq 2$ and at least one of the $a_i$ is at least two,
\item $N \geq 1$ and at least one of the $a_i$ is at least three,
\item $N \geq 0$ and at least one of the $a_i$ is at least four.
\end{enumerate}
Here we use the fact that the power $\Delta^N$ can `slide through' the braid $\beta$ to stand right next to the power $\sigma_k^{a_i}$. The above cases cover many, but not all, positive 3-braids; in particular not braids of the form
$$\beta=\sigma_1^{a_1} \sigma_2^{a_2} \sigma_1^{a_3} \ldots \sigma_2^{a_{2n}},$$
with all $a_i$ equal to~$2$ or~$3$. However, in that case, after removing the term $\sigma_k^{a_n}$ near the middle of $\beta$, we find that the two terms on the left and right of $\sigma_k^{a_n}$ merge to a fourth or higher power of a single generator. After removing that, we find another pair of terms that merge to a fourth or higher power of a single generator. We keep iterating this, until we are left with the positive braid $\sigma_2^{a_{2n}}$, whose length is at most three. The same trick allows us to reduce positive 3-braids of the form
$$\beta=\Delta \sigma_1^{a_1} \sigma_2^{a_2} \sigma_1^{a_3} \ldots \sigma_2^{a_{2n}},$$
with all $a_i=2$, to the braid $\Delta \sigma_2^{a_{2n}}$ of length five. The five remaining positive 3-braids not covered by the above arguments are
$\beta=\Delta^3,\Delta^2,\Delta^2 \sigma_1,\Delta^2 \sigma_2, \Delta^2 \sigma_1 \sigma_2$, all of which have length at most~$9$. By inverting the above procedure, we obtain the desired statement on the structure of positive 3-braids.
\end{proof}

Before continuing with the derivation of linear lower bounds on the signature function, we briefly discuss a geometric interpretation of Proposition~\ref{pentafoil}. The closure of a positive braid $\beta \in B_3^+$ that contains both generators $\sigma_1,\sigma_2$ admits a natural fibre surface in $S^3$, which is obtained from a disc by a finite sequence of positive Hopf plumbings~\cite{Sta}. On the level of braids, positive Hopf plumbing simply means adding a positive generator. Similarly, adding the $n$-th power of a generator to a positive 3-braid corresponds to a plumbing of the fibre surface of a torus link of type $T(2,n+1)$. Therefore, Proposition~\ref{pentafoil} implies that the fibre surface of positive 3-braid links admits a higher order plumbing structure, up to an error of small genus. We call this a pentafoil plumbing structure, in allusion to the trefoil plumbing structure derived in~\cite{BD}, since the knot $T(2,5)$ is the pentafoil knot.

Back to Levine-Tristram signatures, let us mention that Proposition~\ref{pentafoil} implies the known bound
$$\liminf \left\{\frac{\sigma(\beta)}{b_1(\beta)} \middle| \beta \in B_3^+ \right\} \geq \frac{1}{2},$$
for the classical signature invariant $\sigma(\beta)=\sigma_{-1}(\beta)$ (see Stoimenow~\cite{Sto} and Feller~\cite{F1}, who extended this to positive 4-braids). Indeed, adding the $n$-th power of a positive generator to a positive 3-braid $\beta \in B_3^+$ increases the first Betti number $b_1(\beta)$ by $n$, and the signature $\sigma(\beta)$ by at least $n-2$, since we add a direct summand of type $A_{n-1}$ to the symmetrised Seifert form of the closure of~$\beta$. We will now upgrade this to a constant $c>\frac{1}{2}$, for all $\omega \in S^1 \setminus \bar{\Q}$ with $\text{Re}(\omega)<-\frac{1}{2}$. At first glance, this seems impossible, since the family of braids $\beta_n=(\sigma_1^2 \sigma_2^2)^n$ satisfies
$$\lim_{n \to \infty}\frac{\sigma(\beta_n)}{b_1(\beta_n)}=\frac{1}{2}.$$
However, for all $\omega \in S^1 \setminus \bar{\Q}$ with $\text{Re}(\omega)<0$, this limit is strictly bigger than $\frac{1}{2}$. In order to see this, we use the following expression for powers of a positive full twist on 3 strands (compare Section~4.3 in~\cite{FK}):
$$\sigma_2(\sigma_1^2 \sigma_2^2)^{n-1}\sigma_1^2 \sigma_2 \sigma_1^{2n}=(\sigma_1 \sigma_2)^{3n}.$$
Up to a small constant error, the signature function of the two families of braids $\sigma_1^n$ and $(\sigma_1 \sigma_2)^{3n}$ is linear, resp. piecewise linear with a breakpoint at $\omega=\zeta_3=e^{\frac{2 \pi i}{3}}$. Therefore, the resulting signature function for the family of braids $\beta_n=(\sigma_1^2 \sigma_2^2)^n$ is also piecewise linear, with a peak (and breakpoint) of $\frac{2}{3} b_1(\beta_n)$ at $\omega=\zeta_3$, sloping down to $\frac{1}{2} b_1(\beta_n)$ between $\omega=\zeta_3$ and $\omega=-1$ (see also Corollary~4.4 in~\cite{GG} for the value of $\sigma_\omega$ at $\omega=\zeta_3$).

For the rest of this section, we fix $\omega \in S^1 \setminus \bar{\Q}$ with $\text{Re}(\omega)<-\frac{1}{2}$. An inspection of the Levine-Tristram signature function of 2-strand torus links reveals that  $\sigma_\omega(T(2,n))$ is maximal for $n \leq 6$. Adding the $n$-th power of a positive generator to a positive braid $\beta \in B_3^+$ increases the Betti number $b_1(\beta)$ by $\Delta b_1=n$ and the signature $\sigma_\omega(\beta)$ by $\Delta \sigma_\omega$. The assumption $\text{Re}(\omega)<-\frac{1}{2}$ implies

\begin{enumerate}
\item[(i)] $\frac{\Delta \sigma_\omega}{\Delta b_1} \geq \frac{1}{2}$, if $n=4$,

\smallskip
\item[(ii)] $\frac{\Delta \sigma_\omega}{\Delta b_1} \geq \frac{3}{5}$, if $n=5$,

\smallskip
\item[(iii)] $\frac{\Delta \sigma_\omega}{\Delta b_1} \geq \frac{4}{6}$, if $n=6$.
\end{enumerate}

\smallskip
Now let $\beta \in B_3^+$ be a positive braid with Garside normal form
$$\beta=\Delta^N \sigma_1^{a_1} \sigma_2^{a_2} \sigma_1^{a_3} \ldots \sigma_2^{a_{2n}},$$
as above. If $N \geq 5$, then $\beta$ contains the $5$-th power of $\sigma_1$, since
$$\Delta^5=\sigma_1 \sigma_2^2 \sigma_1^2 \sigma_2^2 \sigma_1^2 \sigma_2 \sigma_1^5.$$
Removing the term $\sigma_1^5$ from $\beta$, we obtain a positive braid $\bar{\beta} \in B_3^+$ with smaller first Betti number. Moreover, passing from $\bar{\beta}$ to $\beta$ increases the ratio $\frac{\sigma_\omega(\bar{\beta})}{b_1(\bar{\beta})}$, provided the latter was smaller than $\frac{3}{5}$. Therefore, in order to derive a limit inferior bigger than $\frac{1}{2}$, we only need to consider positive braids $\beta$ with $N \leq 4$. Morevover, since we are not interested in bounded deviations, we may as well completely forget about the term $\Delta^N$ and assume
$$\beta=\sigma_1^{a_1} \sigma_2^{a_2} \sigma_1^{a_3} \ldots \sigma_2^{a_{2n}},$$
with all $a_k \in \{2,3,4\}$. The last assumption is again justified by the discussion on adding the $5$-th generator of a generator to a positive 3-braid. We already dealt with the special case where all the coefficients are two:
$$\liminf \left\{\frac{\sigma_\omega((\sigma_1^2 \sigma_2^2)^n)}{b_1((\sigma_1^2 \sigma_2^2)^n)} \right\}=c_\omega > \frac{1}{2}.$$
In particular, there exists $M \in \N$, depending on $\omega$, with the following property: adding $(\sigma_1^2 \sigma_2^2)^m$ to a positive 3-braid, with $m \geq M$, changes the ratio $\frac{\sigma_\omega}{b_1}$ towards $c_\omega > \frac{1}{2}$.
Therefore, we may assume that the coefficients $a_1,a_2,\ldots,a_{2n}$ of our braid $\beta \in B_3^+$ do not contain a sequence of more than $M$ consecutive $2$'s.

Next, let us suppose $a_k \in \{2,3\}$, for all $k \leq 2n$. Then the procedure described just before Proposition~\ref{pentafoil} shows that, up to a fixed bounded error, $\beta$ is obtained from the empty braid by successively adding $4$-th, $5$-th, or $6$-th powers of positive generators. The proportion of coefficients $a_k=3$ being at least $\frac{1}{M}$, this provides a constant limit inferior strictly bigger than $\frac{1}{2}$ for the ratio $\frac{\sigma_\omega(\beta)}{b_1(\beta)}$ of all these braids, since adding the $5$-th or $6$-th power of a generator improves this ratio towards a value~$\geq \frac{3}{5}$.

Finally, let us discuss the case when one of the coefficients $a_k$ is four. After a suitable conjugation of $\beta$, we may assume $a_n=4$.

If $a_{n-1}+a_{n+1}=7$ or~$8$, then we reduce the braid $\beta$ to the positive braid represented by the sequence
$$a_1,\ldots,a_{n-2},2,a_{n+2},\ldots,a_{2n},$$
by removing first a $4$-th power, and then a $5$-th or $6$-th power of a generator. The latter improves the ratio $\frac{\sigma_\omega(\beta)}{b_1(\beta)}$ towards a value~$\geq \frac{3}{5}$.

If $a_{n-1}+a_{n+1}=5$ or~$6$, then we reduce the braid $\beta$ to the positive braid represented by the sequence
$$a_1,\ldots,a_{n-3},a_{n-2}+a_{n+2},a_{n+3},\ldots,a_{2n},$$
again by removing first a $4$-th power, and then a $5$-th or $6$-th power of a generator. If $a_{n-2}+a_{n+2} \geq 5$, we keep iterating this process, until we are left with a sequence with all coefficients $\leq 4$. Otherwise, $a_{n-2}+a_{n+2} \geq 4$, and we proceed as in the last step, just below.

If $a_{n-1}+a_{n+1}=4$, then again, we reduce the braid $\beta$ to the positive braid represented by the sequence
$$a_1,\ldots,a_{n-3},a_{n-2}+a_{n+2},a_{n+3},\ldots,a_{2n}.$$
However, in this case, we remove two times the $4$-th power of a generator, which does not a priori work in favour of a ratio $\frac{\sigma_\omega(\beta)}{b_1(\beta)}>\frac{1}{2}$. In the unfortunate case when $a_{n-2}+a_{n+2}=4$, $a_{n-3}+a_{n+3}=4$, etc., we keep removing $4$-th powers of a generator, which does not allow any simple conclusion. One more time, it is the upper bound~$M$ on the number of consecutive $2$'s in the initial sequence $a_1,\ldots,a_{2n}$ that comes to rescue in this case: after removing at most $M$ $4$-th powers of a generator, there must eventually appear a pair of coefficients $a_{n-l},a_{n+l}$ with $l \leq M$ and $a_{n-l}+a_{n+l} \geq 5$. As a consequence, the braid~$\beta$ is obtained from the empty braid by adding $4$-th, $5$-th, and $6$-th powers of positive generators, with a fixed lower bound on the proportion of the latter two. This yields the desired lower bound $>\frac{1}{2}$ on the limit inferior of the ratio $\frac{\sigma_\omega}{b_1}$ on $B_3^+$.

\bigskip
\noindent
\texttt{sebastian.baader@math.unibe.ch}

\end{document}